\documentclass[11pt]{amsart}
\usepackage{geometry} 
\usepackage{fancyhdr}
\usepackage{graphicx,
	amssymb,
	amsmath,
	amsthm,
	bm,
	dsfont,
	epstopdf,
	mathrsfs,
	mathabx,
	multirow,
	mathtools,
	dsfont,
	xcolor,
	tabu,
	subfig,
	hyperref}
\hypersetup{
    colorlinks=true,
    allcolors=blue
}

\usepackage{amsopn, amsthm, amsgen, amscd,amsmath, amssymb}
\usepackage[utf8]{inputenc}
\usepackage[english]{babel}
\usepackage{fancybox}
\usepackage[all]{xy}
\usepackage{xcolor}
\usepackage{verbatim}

\newtheorem{lemma}{Lemma}[section]
\newtheorem{theorem}[lemma]{Theorem}

\newtheorem{corollary}[lemma]{Corollary}
\newtheorem{proposition}[lemma]{Proposition}

\theoremstyle{definition}
\newtheorem{definition}{Definition}[section]

\numberwithin{equation}{section}
\theoremstyle{remark}
\newtheorem{Rem}{Remark}[section]
\makeatletter
\def\blfootnote{\gdef\@thefnmark{}\@footnotetext}
\makeatother

\newcommand\R{\mathbb R}
\newcommand\C{\mathbb C}

\newcommand\del{\partial}

\newcommand\Lam{\mathscr F}
\newcommand\nnfootnote[1]{%
  \begin{NoHyper}
  \renewcommand\thefootnote{}\footnote{#1}%
  \addtocounter{footnote}{-1}%
  \end{NoHyper}}

\title{Uniformization of compact foliated spaces by surfaces of hyperbolic type via the Ricci Flow}
\dedicatory{Dedicated to the memory of Marco Brunella.}

\author{Richard Mu\~{n}iz Manasliski\\
        Alberto Verjovsky}

\address{{Centro de Matem\'atica, Facultad de Ciencias},
Igu\'a 4225 esq. Mataojo C.P. 11400, Montevideo, Uruguay}
\email{rmuniz@cmat.edu.uy}

\address{{Instituto de Matem\'aticas, Universidad Nacional Aut\'onma de M\'exico},
Ciudad Universitaria, Delegaci\'on Coyoac\'an, Ciudad de M\'exico. M\'exico}
\email{alberto@matcuer.unam.mx}

\begin{document}

\begin{abstract} We give a new proof of the uniformization theorem of the leaves of a comact lamination by surfaces of hyperbolic conformal type. 
We use a laminated version of the Ricci flow to prove the existence of a laminated Riemannian metric (smooth on the leaves, transversally continuous)
with leaves of constant Gaussian curvature equal to -1, which is conformally equivalent to the original metric.
\end{abstract}

\nnfootnote{\textup{2010} \textit{Mathematics Subject Classification}: 53C12, 53C44}

\maketitle
\section{Introduction}

One version of the Uniformization theorem asserts that any orientable smooth surface $S$ with a Riemmanian metric $g$ admits a \emph{unique} metric $\hat{g}$ 
conformally equivalent to $g$ such that the Gaussian curvature is a constant function on $S$. It implies that given any Riemann surface $\Sigma$ its 
universal covering as Riemann surface obeys the trichotomy: it is either the complex plane $\C$, the Riemann sphere
 $\overline\C$ or the unit disc $\mathbb D$ in the complex plane. In particular every simply connected Riemann surface is isomorphic to the plane, the open unit 
 disc, or the sphere.
 This theorem has been hailed by Lars Ahlfors as
\emph{``perhaps the single most important theorem in the whole theory of functions of one variable;
it does for Riemann surfaces what the Riemann Mapping theorem does for plane regions''}  ~\cite{Ah1} 
(see \cite{Gray} for the history of the Riemann mapping theorem and \cite{SG} for the history of the Uniformization theorem).
This theorem, attributed to Klein, Koebe and Poincar\'e, was proved rigorously and almost simultaneously in 1907 by Paul Koebe \cite{Ko} and Henri Poincar\'e 
\cite{Poinc}. It has been crowned by its important, useful and beautiful generalizations, for instance the theory of quasiconformal mappings, Teichm\"uller 
theory, the Measurable Riemann Mapping theorem and the study of conformal invariants just to cite a few. It has many important applications to other branches 
of mathematics like hyperbolic geometry,  Kleinian groups and complex dynamical systems \cite{Ah3, Ah4, Mi, LM, LMM, BNS, Su}. 

On the other hand, the theory
of foliations and laminations by Riemann surfaces has also a wide spectrum of application, for instance essential laminations and
tight foliations on 3-manifolds have played an important role in the study of these manifolds \cite{GO, Ca1, Ca2} and also in dynamical systems and ergodic 
theory \cite{BGMM, ADMV, LM,Su}. Here we would like to mention the remarkable work by Marco Brunella on holomorphic foliations and their uniformizations 
\cite{Bru, Br1, Br2, Br3, Br4, Br5, Br6, Br7, Br8}. Remember that a lamination is a topological space which is locally a product of a disk
 and a metric space, see definition 2.1. A function defined on a lamination is said to be regular or $C^{\infty}-$laminated (we also denote by $C^{\infty,0}$ the set of this functions) if it is $C^{\infty}$
 restricted to each leaf and its derivatives of all orders (including 0-order) in the leaf direction are continuous in the ambient space. Similarly, a 
 laminated Riemannian metric is a collection of Riemannian metrics, one in each leaf, such that all its leafwise derivatives are continuous 
 on the lamination.

Since compact laminations by surfaces are a generalizations of (a continuous family of) surfaces 
it is natural to ask for a \emph{laminated version} of the uniformization theorem: {\bf Given a compact lamination 
$(M,\Lam)$ by surfaces with a laminated Riemannian 
metric $g$ find a laminated conformal metric $\hat{g}$ which renders each leaf of constant curvature}. 

Each leaf of a lamination has a well defined conformal type, independently of 
the laminated metric we put. That is to say, for any Riemannian metric on the tangent bundle to the leaves which varies continuously in the 
lamination, each leaf will have the same conformal character (elliptic, hyperbolic or parabolic). We can easily construct examples where 
hyperbolic and parabolic leaves are mixed in a non trivial way. Of course our question only makes sense if the lamination is by leaves 
of the same conformal type since the leaves' curvature is a continuous function. Certainly there always exists a metric in each leaf having 
constant curvature, but it is not necessarily regular as laminated metric. In the case where all the leaves are elliptic Reeb's stability 
theorem implies the triviality of the lamination, i.e. the lamination is a continuous family of spheres, and using Ahlfors-Bers theory 
it follows that the uniformization (and all its tangential derivatives) depends continuously on the parameters. For laminations by parabolic 
surfaces the answer is negative in general, and the most we can do is to aproximate the uniformization as was done in \cite{Ghys}. In 
\cite{Candel} the regularity of the uniformization in this case is proved under certain topological restriction. When the lamination 
is of hyperbolic type the answer is positive, as was proved by Alberto Candel in \cite{Candel}. In this last case the proof is not simple 
and relies on some analytical techniques carefully applied to a foliation. The introduction of the uniformization metric and the proof 
of its lower semicontinuity was acheived by the second author in \cite{Alberto}, then Candel proved in \cite{Candel} that it is also 
upper semicontinuous and that all its leafwise derivatives are continuous as well. The aim of the present work is to give a new 
proof of this result from a more geometric viewpoint.

As it is widely known the uniformization theorem for compact surfaces can be proved using the celebrated Hamilton's Ricci flow. This is 
acheived for any compact Riemannian surface in a series of papers by Richard Hamilton \cite{Ham2}, Bennet Chow \cite{Chow} and  
Xiuxiong Chen, Peng Lu and Gang Tian \cite{ChenLuTian}. For noncompact surfaces it is an open question when the uniformization can 
be obtained in this way. There are some works in this direction, see for example \cite{Isen-Mazz-Sesum,Albin-Aldana-Rochon} and 
references therein, assuming somewhat restrictive
conditions on the metric. Inspired by this circle of ideas, it sounds natural to try to obtain the uniformization theorem for compact laminations via 
the Ricci flow. Even if the leaves are not necessarily compact, and for non compact leaves the metric will not satisfy the condition for long time existence 
and convergence of the solutions to  the Ricci flow, this difficulty can be overcome thanks to the compactness of the total space where the lamination lives.

The use of geometric flows for the study of foliations on smooth manifolds is not new, and there are many works in this sense by 
Vladimir Rovenski and Pawel Walczac  \cite{RovWal,RovWal1,RovWal2,Rovenski1,Rovenski2} for example. 
But in all the works we know it is the ``transversal'' (or extrinsic) Ricci flow which is considered and mostly in codimension one foliations; the metric on the leaves is fixed 
and what changes with the flow is the transversal metric. In the present paper what we use is the ``tangential'' (or intrinsic) Ricci flow for rank 2 foliations with no 
restriction on the transversal; what changes with the flow here is the metric on each leaf. More concretely we consider the following Cauchy 
problem
\begin{equation}
\left\{\begin{aligned} \frac{\partial g(t)}{\partial t} &=(r-R(t))g(t),\\
                       g(0)&=g_0
       \end{aligned}\right.
\end{equation}
which is a kind of ``normalized'' Ricci flow equation on each leaf, where $g_0$ is a given laminated metric, $R(t)$ is  the function that assigns to each point $x$ the 
curvature of the leaf that passes through $x$ with respect to the metric $g(t)$, and $r$ is a constant to be chosen conveniently. (In the case of compact surfaces this constant $r$ is usually taking to be the average of the scalar curvature). Note the 
important fact that the conformal class is invariant under this flow.

General theorems by Wan-Xiong Shi \cite{Shi}, and Bing-Long Chen and Xi-Ping Zhu \cite{Chen-Zhu}, guarantee  short time existence and uniqueness
 of solution to the Ricci flow on a noncompact manifold for a complete initial metric of bounded curvature. Note that for a metric on a compact 
 lamination each leaf is necessarily complete and the curvature of the leaves is a continuous function and therefore bounded. The first problem is to show that all the solutions obtained by 
 running the Ricci flow on each leaf patch together to give a one parameter family of laminated metrics, which proves short time existence 
  for 1.1 as a laminated problem. The second step is to show long time existence of solutions, for what we need some {\it a priori} bounds on
  the curvature and its 
  derivatives. To obtain these bounds from the maximum principle we have to start the flow with an initial condition having negative curvature 
  at all points; metrics of this type can always be chosen within the conformal class of any laminated metric, 
  provided the lamination is of hyperbolic type, thanks to a simple argument attributed to Éthienne Ghys. 
  Finally, the {\it a priori} bounds give convergence of the solution $g(t)$ when $t\to+\infty$ to a metric in each leaf, and again we have to 
  prove laminated regularity of this collection of leafwise metrics which is obtained in the limit.

To finish this introductory section we now briefly describe the content of the paper.
In Section \ref{Ricci_Flow} we introduce the laminated Ricci flow and develop the  
{\it a priori} curvature estimates with respect to the flow. In Section \ref{covering_tubes} we prove the existence of covering tubes, which 
are a technical device introduced by Ilyashenko \cite{Ily1,Ily2} and also used by Brunella \cite{Bru}, that allows us to 
study the flow as a parametrized family of flows on surfaces. Then in Section
\ref{existence} uniqueness and continuity of the solution to the Ricci flow is proved. Finally, in \ref{uniformization} we 
recast the uniformization theorem in the spirit of Richard Hamilton.
The proof relies on the existence of covering tubes and uniqueness of solutions for the Ricci flow on manifolds. It is proved that a covering tube of 
the form $\Phi:\Sigma\times{D}\to{M}$ , where $\Sigma$ is a transversal and $D\subset\C$ the unit disk,  always exists. Taking the pull-back $\Phi^*g$ of the laminated 
metric we solve the problem in the tube to prove laminated regularity. 

%In Section \ref{Concluding_Remarks} \emph{Concluding Remarks} we mention the fact that a uniformization theorem
%for parabolic foliations is not always possible due to a counterexample by \'Etienne Ghys \cite{Ghys}. Also in the noncompact case
%it is not always possible to find covering tubes (like in the Reeb component, for instance) due to a phenomenon related to the existence of vanishing cycles; 
%we refer to the paper by Marco Brunella \cite{Br5}.

\section{Preliminaries: laminations and laminated metrics.}

\begin{definition} A {\em lamination of rank m}, $(M,\Lam)$, is a second countable, locally compact, metrizable topological space $M$ endowed with an atlas $(U_\alpha, \varphi_\alpha)$ 
such that:
\begin{itemize}
 \item[(1)] Each $\varphi_\alpha$ is a homeomorphism from $U_\alpha$ to a product $D_\alpha\times \Sigma_\alpha$, where $D_\alpha$ 
 is a ball in the Euclidean space $\R^m$ and 
 $\Sigma_\alpha$ is a topological space.
 \item[(2)] Whenever $U_\alpha\cap U_\beta \neq \emptyset$, the change of coordinates 
 $\varphi_\beta\circ\varphi_\alpha^{-1}$ is of the form
 $$(z,\zeta)\mapsto (\lambda_{\alpha\beta}(z,\zeta), \tau(\zeta)),$$ where $\lambda_{\alpha\beta}$ is smooth in the $z$ variable. If
 the $\lambda_{\alpha\beta}$ preserve a fixed orientation of the m-ball we say that the lamination is oriented.
\end{itemize}

The sets $\varphi_\alpha^{-1}(D_\alpha\times\{\zeta\})$ are called {\em plaques}. Condition (2)
says that the plaques glue together to form smooth manifolds, 
called {\em leaves}, which are ``immersed" in $M$. We will say that the lamination is compact if the total space $M$ is compact. 

When the $\Sigma_\alpha$ are open subsets of $\R^n$ and the $\varphi_\alpha$ are smooth, the space $M$ is a manifold and the laminated
structure $\Lam$ is called a {\em smooth foliation}.

\end{definition}

We say that a function $u\colon M\rightarrow\R$ is laminated of class $C^{k,0}$ and we write $u\in C^{k,0}(M,\Lam)$, or simply $u\in C^{k,0}$ 
if the context is clear, if $u$ is a continuous function on $M$ which when restricted to a leaf is of class $C^k$ in the usual sense and 
all its tangential derivatives up to order $k$ are continuous functions on $M$. Similarly we can also define a laminated map $\psi\colon(M_1,\Lam_1)\rightarrow(M_2,\Lam_2)$, between two laminations, to be of class $C^{k,0}$ if: it is a continuous function from $M_1$ to $M_2$, it  sends leaves to leaves, it is of class $C^k$ restricted to each leaf and its derivatives up to order $k$ in the leaves' direction are continuous.  

Several vector and fiber bundles over $M$ can be defined in a natural way using the fact that the $\lambda_{\alpha\beta}$ are smooth
in the variable $z$. These include the tangent bundle $T\Lam$ to the lamination, tensor bundles, frame bundles, etc. Locally, their fibers
vary continuously in the smooth topology of $D_{\alpha}$ parametrized by the transversal $\Sigma_{\alpha}$ of the laminated structure. In the same manner we can also talk about pullbacks with respect to laminated maps. We refer to ~\cite{Moore} and ~\cite{Candel-Conlon} for all the basic notions on the theory of laminations.

From now on we will consider only laminations of rank 2, i.e. by surfaces, unless otherwise stated.

A {\em laminated Riemannian metric} $g$ is a laminated tensor which restricted to each leaf is a ($C^{\infty}$) Riemannian metric 
 on the leaf. 
Remark that, if
the space $M$ is compact, two Riemannian metrics $g$ and $g'$ induce quasi-isometric metrics on any given leaf. 
This allows us to
speak about {\em hyperbolic} or {\em parabolic} leaves, independently of the metric. More precisely:

\begin{definition} \label{definition:hyperbolic_leaf}
 Let $(M,\Lam)$ be a compact lamination by surfaces, and let $L$ be a leaf of $\Lam$. Choose any laminated 
 Riemannian metric $g$ on $(M,\Lam)$.
 Let $\tilde L$ be the universal cover of $L$, which is a Riemannian manifold when endowed with the pullback of the restriction 
 of $g$ to $L$. For $x\in\tilde L$ and $r>0$, let $A(r)$ be the area of the disk of radius $r$ and centered at $x$. 
 We say that the leaf $L$ is {\em hyperbolic} if $A(r)$ grows exponentially as a function of $r$, or equivalently $\tilde L$ is conformally equivalent to the Poincaré disk. We say that $L$ is {\em parabolic} if it is not a sphere and $A(r)$ has polynomial growth, or equivalently if $\tilde L$ is conformally equivalent to flat $\R^2$. Finally, $L$ is {\em elliptic} if $L$ is diffeomorphic to  a sphere.
\end{definition}

As explained above, this definition does not depend on $g$. It clearly does not depend on the choice of the point $x$ either.

On an \emph{oriented} laminated space a laminated Riemannian metric $g$ determines a conformal structure on every leaf, that is,
it turns every leaf into a Riemann surface. This follows using local isothermal coordinates or the natural almost complex structure
which is compatible with the laminated metric and the orientation. 
A leaf $L$ is hyperbolic in the sense of Definition \ref{definition:hyperbolic_leaf} if and only if it is a hyperbolic
Riemann surface for any choice of $g$. In this case, it can be uniformized by the unit disk.

When all leaves are hyperbolic, the uniformization maps of individual leaves vary continuously from leaf to leaf. 
More precisely, the following {\em Uniformization theorem}
holds (see \cite{Candel}):

\begin{theorem}\label{theorem:uniformization}
 Let $(M,\Lam)$ be a compact lamination by hyperbolic surfaces endowed with a laminated Riemannian metric $g$. Then there is 
 a laminated Riemannian metric $g'$ which is conformally equivalent to $g$ and for which every leaf has constant curvature -1.
% (Remark that $g'$ has a continuous variation in the smooth topology in the direction transverse to $\Lam$.)
\end{theorem}

The objective of this paper is to give a new proof of this theorem using some of the machinery from the Ricci flow.
%Our objective is to give a proof of this theorem using the Ricci flow. In Section 2 of this paper we introduce the laminated Ricci flow and develop 
%the {\it a priori} curvature estimates with respect to the flow.  In section 3 we define {\it covering tubes} in the sense of Ilyashenko
%\cite{Ily1,Ily2} \cite{Bru} and prove the existence of such tubes in the case of laminations by surfaces of hyperbolic type. 
%We show the nonexistence of vanishing cycles for laminations of the above type. This was shown in \cite{ADMV}  but we give another 
%proof. Using the covering tube and taking the pull-back of the laminated metric we show the existence and uniqueness of the solution to the flow 
%and the continuity of the limit metric when $t\to\infty$.

\section{Laminated Ricci flow}\label{Ricci_Flow}

\subsection{The Ricci flow and the maximum principle}

Let $(M,\Lam)$ be a compact lamination and $g_0$ a laminated metric on it. We can 
consider the ``normalized laminated Ricci flow'' as the evolution of the metric under 
the equation
\begin{equation}\frac{\partial g}{\partial t}=(r-R)g,\end{equation}
with initial condition $g(0)=g_0$;
here $R$ is the scalar curvature of the leaves and $r$ is a constant 
(to be fixed conveniently). Let us denote by $R_0$ the scalar curvature of the leaves 
with respect to the metric $g_0$. Since $M$ is compact the leaves are complete and 
moreover $R_0$, being a 
continuous function on $M$, is bounded. 
From this it is possible to conclude that there exists $\epsilon >0$ such that for each $t$ in a time interval 
$[0,\epsilon)$ there is a solution $g(t)$ to the  Ricci flow equation; for $g(t)$ to be a solution to $(1)$ on $(M,\Lam)$ it 
has to vary continuously in the transverse direction, a fact that is essentially a consequence 
of the continuous dependence of the solution to $(3.1)$ with respect to the initial 
condition. We postpone the proof of these facts to Section 5 and now we establish some {a priori} bounds on the curvature function
that can be deduced from the maximum principle and which are essential for the long time existence of solutions.

It is easy to see that the 
curvature of a family of metrics $g(t)$ satisfying (1) evolves under the 
diffusion-reaction equation ~\cite{Ham2}:
   \begin{equation}  
    \frac{\partial R}{\partial t}=\Delta R+(R-r)R.
   \end{equation}
Here $\Delta$ denotes the Laplacian in the leaf direction 
(with respect to $g(t)$), i.e. we 
consider the above equation on each leaf.

 An important fact in the two dimensional case is that equation 
(3.1) leaves invariant the conformal class of the initial metric $g_0$, hence we 
can write the evolution as an evolution of a single function $u$. More 
precisely, by writing $g=e^ug_0$ for a metric in the conformal class of $g_0$,
 we have that under the Ricci flow $u$ evolves according to
  \begin{equation}
   \frac{\partial u}{\partial t}=r-R=\Delta u-e^{-u}R_0+r=e^{-u}(\Delta_0 u-R_0)+r.
  \end{equation}
  We denote here by $\Delta_0$ the Laplacian associated to $g_0$ and we
 use the well known fact that $\Delta_{e^ug_0}=e^{-u}\Delta_0$.

 Thanks to the compactness of $M$ we can use the maximum principle to 
control the evolution of the geometric quantities under the Ricci flow. We 
state and prove here the versions of the maximum principle that we will use in the sequel. The proofs are the same as 
for compact surfaces since we only need leafwise differentiation (the arguments 
are taken form ~\cite{Chow-Knopf} ), but we have decided to include them for the
reader's convenience.

\begin{proposition} Let $v\colon M\times[0,T]\rightarrow\R$ be a function which is $C^{2,0}$ in $x\in M$ and $C^1$ in $t$ such that
                    $$ \frac{\partial v}{\partial t}\leq\Delta v+\beta v+b,$$
                   where $\beta\colon M\times[0,T]\rightarrow\R$ satisfies $\beta(x,t)\leq-C<0$ for 
                   a constant $C$, and $b$ is a non negative constant. If $v(x,0)\leq0$ for all $x$ in $M$,
                   then $v(x,t)\leq b/C$ for all $x\in M$ $t\in[0,T]$.  
\end{proposition}

\begin{proof}
 Define, for a positive $\varepsilon$, $F\colon M\times[0,T]\rightarrow\R$ by
$$F(x,t)=e^{Ct}(v(x,t)-b/C)-\varepsilon t-\varepsilon.$$
It is enough to prove that $F$ is everywhere negative. Suppose by the contrary
 that $F$ vanishes at some point, then by compactness there will be a first time 
$t_0$ for which $F$ vanishes. Hence there will be a point $x_0\in M$ such that:
 \begin{itemize} 
      \item[$\bullet$] $F(x_0,t_0)=0$, and 
      \item[$\bullet$] $F(x,t)<0$ for all $x\in M$ if $t<t_0$
 \end{itemize}
 This implies that $\frac{\partial F}{\partial t}(x_0,t_0)\geq 0$. But on the 
other hand we have
\begin{align*}
 \frac{\partial F}{\partial t}(x_0,t_0) &\leq\Delta F(x_0,t_0)+\left(\frac{b}{C}e^{Ct_0}+(1+t_0)
 \varepsilon\right)(C+\beta)-\varepsilon\\
                              &\leq\Delta F(x_0,t_0)-\varepsilon,\end{align*}
and since $F(\cdot,t_0)$ has a maximum at $x_0$ when restricted to the leaf through $x_0$ we 
have $\Delta F(x_0,t_0)\leq 0$ and we arrive at a contradiction. 
\end{proof}

\begin{Rem}
 Notice that if we take $b=0$ in the above proposition, $C$ can be 
negative and the same argument shows that we have $v(x,t)\leq0$ for all $x\in M$ $t\in [0,T]$. In other words, if $\beta$ is bounded and we start with a nonpositive initial condition then any subsolution remains nonpositive.
\end{Rem}

\begin{proposition} Let $(M,\Lam)$ be a compact lamination and $g(t)$ be a one parameter family
               of Riemannian metrics on $(M,\Lam)$. Suppose that 
               $v\colon M\times[0,T]\rightarrow\R$ is a function which is $C^{2,0}$ with respect to $M$ 
               and $C^1$ with respect to $t$, and such that
               $$\frac{\partial v}{\partial t}\leq\Delta v+F(v)$$
               where $F\colon\R\rightarrow\R$ is a locally Lipschitz function. 
               Suppose $v(x,t)\leq c$ for all $x \in M$, then 
               $v(x,t)\leq\varphi(t)$ where $\varphi$ is the unique solution
               to the Cauchy problem
               $$\left\{\begin{aligned} \dot\varphi(t)&=F(\varphi(t))\\
                                          \varphi(0)&=c.\end{aligned}\right.$$
\end{proposition}

\begin{proof}
 Take $w=v-\varphi$, then
 $$\frac{\del w}{\del t}=\frac{\del v}{\del t}-\dot{\varphi}\leq \Delta v+F(v)-F(\varphi).$$
 Since $F$ is locally Lipschitz, and $M$ is compact there exists a constant $C$ such that 
 $$|F(v)-F(\varphi)|\leq C|v-\varphi|.$$
 Therefore
 $$\frac{\del w}{\del t}\leq \Delta v+C\,\text{sign}(w)w,$$
 and taking $\beta=C\,\text{sign}(w)$, Proposition 3.1 adapted for $b=0$ implies that
 $v\leq\varphi$ (see remark 3.1).
\end{proof}

 We also have, reversing inequalities, the corresponding propositions for supersolutions (in Proposition 3.1 
we must change the sign of the constant $b$).

\subsection{Negative curvature at all points.}\label{negative_curvature}

 If the scalar curvature of $(M,\Lam,g_0)$ is negative at all points we can use the 
 maximum principle to assure long time existence for the Ricci flow on each leaf. Suppose 
$R_0(x)<0$ for all point $x\in M$. By compactness we have $R_{min}\leq R_0(x)
\leq R_{max}<0$ for all $x\in M$. Suppose $R_{min}\neq R_{max}$  and choose 
a constant $r\in(R_{min},R_{max})$. Then,
 using the maximum principle we have the following proposition.

\begin{proposition}\label{proposition:curvature bounds} Let $g(t)$ the solution to the flow
                    $$\frac{\partial g}{\partial t}=(r-R)g,\quad g(0)=g_0$$
                    defined on a time interval $[0,T)$. If $R_0(x)<0$ for all
$x\in M$, then there exist a positive constant $C$ such that
          $$r-Ce^{rt}\leq R(t)\leq r+Ce^{rt},$$
for all $t\in[0,T)$. Moreover for each positive integer $k$, there
  exists a constant $C_k$ such that
  $$|\nabla^k R|^2\leq C_ke^{\frac{r}{2}t}$$
  for all $t\in[0,T)$.
\end{proposition}

\begin{proof}
 As we have seen, the evolution of the scalar curvature is given by equation (3.2)
 $$\frac{\partial R}{\partial t}=\Delta R+(R-r)R.$$
 Taking $F\colon\R\rightarrow\R$ given by $F(s)=(s-r)s$ we are in the hypothesis to apply
 Proposition 3.2 to bound $R$ above and below. The solution to the Cauchy problem
 $$\left\{\begin{aligned} \dot\varphi(t)&=F(\varphi(t))\\
                                          \varphi(0)&=c\end{aligned}\right.$$
for $r\neq 0$, $c\neq0$ is given by
$$\varphi(t)=\frac{r}{1-(1-\frac{r}{c})e^{rt}},$$
then
$$R\geq \frac{r}{1-(1-\frac{r}{R_{\text{min}}})e^{rt}}\geq r+(R_{\text{min}}-r)e^{rt}$$
and
$$R\leq\frac{r}{1-(1-\frac{r}{R_{\text{max}}})e^{rt}}\leq r+(R_{\text{max}}-r)e^{rt},$$
and the first part of the proposition is proved. The bounds on the derivatives
are also a consequence of the maximum principle and we refer to
~\cite[Proposition~5.27]{Chow-Knopf} for a proof.

\end{proof}
\begin{corollary}
 Under the same hypotheses as in the previous proposition there exists a constant $C$ which depends only on $g_0$ and $r$ such that
 $$\frac{1}{C}g_0\leq g(t)\leq Cg_0,\quad\forall t\in[0,T).$$
\end{corollary}
\begin{proof}
 If we write $g(t)=e^{u}g_0$, it is equivalent to show that $u$ is a bounded function. But $\frac{\partial u}{\partial t}=r-R$, then 
 $$u(x,t)=\int_0^tr-R(x,s)ds,$$
 and applying the bound on the curvature obtained in Proposition 3.3 we have
 $$|u|\leq\int_0^t|r-R|ds\leq \frac{C}{r}(e^{rt}-1)\leq -\frac{C}{r}.$$
 
\end{proof}

 The above proposition together with the bound on the metric
  $$\frac{1}{C}g_0\leq g(t)\leq Cg_0,$$
which is valid as long as the solution exists, implies long time existence of the 
solution on each leaf as in the case of a compact surface 
(see ~\cite[chapters 5 and 6]{Chow-Knopf}). Moreover, the limit metric  $g_{\infty}=
\lim_{t\to\infty}g(t)$ exists for each leaf and has constant scalar curvature $r$ . This can be 
proved using the same arguments that are used in the case of a compact surface, since
they only require differentiation along the leaves. Notice that we have a collection of metrics, one for each leaf, which renders 
each leaf of constant curvature r, the problem is to show regularity as a laminated metric (i.e., continuity of the function $u$ and its derivatives).

 In view of the previous argument it might seem possible to prove
 the uniformization theorem, for a lamination by hyperbolic surfaces, via the Ricci flow.
In that case we need to start with an arbitrary metric $g_0$ whose curvature has possibly varying sign. In this 
situation the 
result is not an immediate consequence of the proof for compact surfaces because it 
needs the Hodge theorem, and we do not have an appropriate laminated version.   
We can overcome this difficulty thanks to an argument attributed to \'Etienne Ghys, that uses the Hahn-Banach
theorem to show 
the existence of a metric of strictly negative curvature in the conformal class
of any given metric on a compact lamination by hyperbolic surfaces. We can therefore
use that result and then apply the Ricci flow to prove the existence of
a uniformizing metric.

\begin{theorem}[Ghys]\label{lemma:etienne}
 Let $(M,\Lam)$ be a compact lamination by surfaces of hyperbolic type. Then, in each conformal class of $(M,\Lam)$, 
 there exists a  Riemannian laminated metric in such a way that the leaves of $\Lam$ have negative curvature at each 
 of its points.
\end{theorem}

For a proof of this theorem see ~\cite{Seba1}, and particularly ~\cite[Theorem~6.5]{Seba2}.

\section{Covering tubes.}\label{covering_tubes}
Following a concept defined by Il'yashenko ~\cite{Ily1,Ily2}, also used by Brunella in 
~\cite{Bru} we define the notion of what we call {\it covering tube}. The basic idea is to obtain a kind
of ``flow box'' which is saturated by leaves. Its existence, in the case of foliated manifolds,  relies on the non existence
of vanishing cycles and can be constructed by gluing together the universal coverings of the leaves based on each of the points of a given transversal. For a lamination by surfaces of hyperbolic type, 
using \ref{lemma:etienne} and the Hadamard theorem, we can show the existence
of covering tubes by taking the exponential map on each leaf.

\begin{definition}\label{definition:coveringtube}
 Let $(M,\Lam)$ be a lamination and 
 $\Sigma$ be a transversal. A {\it covering tube} with respect to $\Sigma$ 
 is a lamination $U_{\Sigma}$ such that:
 \begin{enumerate}
  \item the laminated structure on $U_{\Sigma}$ is given by a continuous fibration $\psi\colon U_{\Sigma}\rightarrow\Sigma$; 
  \item there exists a laminated immersion (local laminated diffeomorphism) $\Phi\colon U_{\Sigma}\rightarrow M$;
  \item there exists a 
  section $\sigma\colon\Sigma\rightarrow U_{\Sigma}$ such that $\Phi\circ\sigma=\text{id}_{\Sigma}$ and  $(\psi^{-1}(\zeta),\sigma(\zeta))$
  is identified with the universal covering of the leaf $L_{\zeta}$ based at $\zeta$ via $\Phi\big|_{\psi^{-1}(\zeta)}$.
  %;such that %$\sigma\circ\Phi=\text{id}_{\Sigma}$ and %which sends each fiber of $\psi$ to the
  %corresponding leaf as universal covering.
 \end{enumerate}

 \begin{lemma}
  If $(M,\Lam)$ is a lamination by surfaces of hyperbolic type and $\Sigma$ is any transversal, then there exists
  a covering tube with respect to $\Sigma$.
 \end{lemma}

 \begin{proof}
  Since we are assuming that all leaves are of hyperbolic type, by \ref{lemma:etienne} there exists a Riemannian 
  metric on $M$ with respect to which all leaves have negative curvature. Then, thanks to a theorem of Hadamard, 
  the exponential map on each leaf based at any point is a covering map. Take any transversal $\Sigma$ 
  to $\Lam$ and define $U_{\Sigma}:=T\Lam|_{\Sigma}$. Given that geodesics depend continuously on initial conditions and 
  parameters, the map
  $$\Phi\colon U_{\Sigma}\rightarrow M$$
  given by $\Phi(\zeta,v):=\text{exp}_{\zeta}(v)$ is of class $C^{\infty,0}$ and it clearly is a local laminated diffeomorphism. Therefore, 
  the lamination $U_{\Sigma}$ is a covering tube of $(M,\Lam)$; the submersion $\psi$ of the definition is the natural 
  projection $U_{\Sigma}\rightarrow\Sigma$ and the section $\sigma$ is the zero section. 
 \end{proof}

\end{definition}

\section{Existence and uniqueness of solutions.}\label{existence}
Since on each leaf we are running the Ricci flow with an initial condition of bounded geometry, 
by Shi's theorem there exists a solution on each leaf which remains with bounded curvature as long as it 
is defined, and by ~\cite{Chen-Zhu} it is also unique. Moreover, the lifetime of each maximal solution is bounded below by a constant depending only 
on the curvature of the initial metric (this constant is a decresing funtion of the maximum of the absolute value of the initial curvature) and then there exists a positive time for which the solutions on
all leaves are simultaneously defined. Therefore, collecting all that solutions we have a function $u\colon M\times[0,T]\rightarrow\R$ defined for some positive $T$ such that $g(t)=e^{u(\cdot,t)}g_0$ solves the Ricci flow equation on each leaf, starting with the metric $g_0$. This is not necessarily a solution to our problem since $u$ is not necessarily continuous in $M$.
 Hence, to prove the existence
of a solution to our problem we must show that $u_t=u(\cdot,t)$ belongs to $C^{\infty,0}(M,\Lam)$ for each $t$, and this essentially 
consists in proving continuous dependence, in the 
$C^{\infty}$ topology, of the Ricci Flow with respect to parameters. When we look at a flow box $D_{\alpha}\times\Sigma_{\alpha}$ we can think the equation for $u$ as an equation in the disk $D_{\alpha}$ depending on the point in the transversal $\Sigma_{\alpha}$ (parameters). We present a proof which gives 
 continuous dependence by adapting Shi's proof
to our context; and taking advantage of the particular situation of having two dimensional leaves.
For this we use the covering tube
to trivialize the lamination. 

Let us first state the theorem we want to prove.

\begin{theorem}\label{theorem:continuity}
 Let $g_0$ be a laminated Riemannian metric of nonconstant negative curvature on a compact surface lamination $(M,\Lam)$.
 Let $g(t)=e^{u_t}g_0$ be the leafwise solution to the normalized Ricci flow equation (3.1) on $(M,\Lam)$, with a constant $r\in(R_{\text{min}},R_{\text{max}})$
 and initial condition $g_0$. Then the function $u_t=u(\cdot,t)$ belongs to $C^{\infty,0}(M,\Lam)$.
\end{theorem}

The strategy to prove the theorem, after taking pullback to a covering tube, is to use well known results about 
parabolic partial differential equations to prove existence and continuous dependence with respect to parameters, as well as Shi's theorem. 
%The situation here is simpler than in arbitrary dimension since %the Ricci  
%flow on a surface is strictly parabolic, and the solutions on compact sets are given by a kernel.

Let $\Phi\colon U_{\Sigma}\rightarrow M$ be a covering tube. Taking pullback with respect to $\Phi$ of the metrics $g(t)$ we have
\begin{align*}
    \Phi^{\ast}g(t)&=\Phi^{\ast}e^{u_t}g_0=e^{u_t\circ\Phi}\Phi^{\ast}g_0;\\
    \Phi^{\ast}R_t &=R_t\circ\Phi;\\
    \intertext{and}
\Phi^{\ast}\Delta_{g(t)} &=\Delta_{\Phi^{\ast}g(t)}=e^{-u_t\circ\Phi}\Delta_{\Phi^{\ast}g_0}.
\end{align*}
Let us put $\tilde{u}\colon U_{\Sigma}\times[0,T]\rightarrow\R$ given by $\tilde{u}(x,t)=u(\Phi(x),t)$,
 $\tilde{\Delta}_0=\Phi^{\ast}\Delta_0$, and $\tilde{R}_0=R_0\circ\Phi$. Then $\tilde{u}$ satisfies the equation
 \begin{equation}
     \left\{\begin{aligned}
      &\frac{\partial\tilde u}{\partial t}-e^{-\tilde u}\tilde{\Delta}_0\tilde{u}=r-\tilde{R}_0e^{-\tilde u},\\
      &\tilde{u}_0=0.
     \end{aligned}\right.
 \end{equation}
Reciprocally, in view of uniqueness of the solutions to the Ricci flow equation, we have that any function which satisfies (5.1) necessarily is the lift of a function $u$ which satisfies (3.3) since the equation is invariant under leafwise deck transformations. 

If we take a covering tube of the form $U_{\Sigma}=\R^2\times\Sigma$, then problem (5.1) can be seen as a family of quasilinear parabolic problems parametrized by $\zeta\in\Sigma$. As was mentioned  before, we know there exists a unique solution on $\R^2$ for each $\zeta$ in $\Sigma$. To show continuity of the family of solutions with respect to $\zeta$ we first consider the problem on a set of the form $\Omega\times\Sigma$ for a bounded open set $\Omega\subset\R^2$ with smooth boundary, and then we let the boundary  go to infinity.

If $\Omega\subset\R^2$  we will denote by  $C^1([0,T],C^{\infty}(\Omega))$ the set of functions $f\colon[0,T]\times \Omega\rightarrow\R$ which are $C^1$ with respect to $t\in[0,T]$ and $f(t,\cdot)\in C^{\infty}(\Omega)$.

We are going to use the following proposition which is a particular case of ~\cite[Theorem 3.2]{Shi}; and which is a consequence of well known results  from the theory of quasilinear parabolic partial differential equations (~\cite [Chapter 6]{Lady}).

\begin{proposition} 
 Let $\Omega\subset\R^2$ be a bounded open subset with smooth  boundary.
 Fix a Riemannian metric on $\R^2$, let $\Delta_0$ be the corresponding Laplacian  and let $R_0$ be the corresponding scalar curvature. Then, there exist $T>0$, which only depends on $\underset{\bar\Omega}{\text{max}}\,|R_0|$, and a unique
 solution $v\in C^1([0,T],C^{\infty}(\bar{\Omega}))$ to 
 the following problem
 \begin{equation*}
 \left\{\begin{aligned}
         &\frac{\partial v}{\partial t}-e^{-v}\Delta_0 v=r-e^{-v}R_0\\
         &v(0,z)=0\quad \forall z\in \Omega\\
         &v(t,z)=0\quad \forall z\in\partial\Omega,
        \end{aligned}\right.,
  \end{equation*}
  where $r$ is any constant.
\end{proposition}

Before we proceed to the proof of Theorem 5.1 let us state the following lemma which is a corollary of the previous proposition.

\begin{lemma}
Let $\Omega\subset\R^2$ be a bounded open subset with smooth boundary, and let $\Sigma$ be a compact topologial space. Suppose that we have a continuous family  of Riemannian metrics on $\R^2$ parametrized by $\Sigma$, and let us denote by $\Delta_0$ and $R_0$ the corresponding Laplacian and scalar curvature with respect to each metric in the family $($both depend on $\zeta\in\Sigma$ of course$)$. Suppose in addition that $R_0$ is negative. Then the function $v\colon[0,T]\times\Omega\times\Sigma\rightarrow\R$, which is defined in some positive time interval $[0,T]$, such that $v(\cdot,\cdot,\zeta)$ is the solution to the problem
\begin{equation}
 \left\{\begin{aligned}
         &\frac{\partial v}{\partial t}-e^{-v}\Delta_0 v=r-R_0e^{-v}\\
         &v(0,z,\zeta)=0\quad \forall z\in \Omega\\
         &v(t,z,\zeta)=0\quad \forall z\in\partial\Omega,
        \end{aligned}\right.
  \end{equation}
 $($where $r$ is an arbitrary constant in the interval $(R_{min},R_{max})$ as in  Section 3 $)$ for $\zeta\in\Sigma$, is continuous.
 \end{lemma}

\begin{proof}
 Since $|R_0(z,\zeta)|$ is continuous by hypothesis, it has a maximum $k_0$ which is then an upper bound for the absolute value of the scalar curvature of each plaque $\Omega\times\{\zeta\}$. This implies that there exists  a positive $T$ such that solutions to (5.2) for all $\zeta$ are simultaneously defined in a time interval $[0,T]$ (for a $T$ that depends only on $k_0$). Therefore, we have a well defined function $v\colon[0,T]\times\Omega\times\Sigma\rightarrow\R$ which solves problem 5.2 for each $\zeta\in\Sigma$. Certainly $v$ is continuous with respect to $t$ and $z$, so all that is left to prove is continuity with respect to $\zeta$.  
 
 Let us first show that $v$ is a bounded function. The maximum principle is valid in each plaque for 5.2, and  analogously as we have proved for a compact lamination we have
 for each $\zeta\in\Sigma$
 $$|r-R(t,z,\zeta)|\leq Ce^{rt},\quad\forall(t,z)\in[0,T]\times\Omega,$$
 for a constant $C$ that does not depend on $\zeta$. Hence, as in Corollary 3.4 we have that $v$ is bounded.
 
 Now, to see that $v$ is continuous with respect to $\zeta$ we are going to use again the maximum principle in each plaque. Lets us take $\hat{\zeta}\in\Sigma$ and $\epsilon>0$. For a function $u(t,z,\zeta)$ define $\hat u(t,z,\zeta):=u(t,z,\hat \zeta)$  and let us take $\omega=(v-\hat v)^2$. Then, using equation 5.2 we have
 $$\frac{\partial\omega}{\partial t}=2(v-\hat v)(e^{-v}\Delta_0v-e^{-\hat v}\widehat{\Delta_0v}+\hat R_0e^{-\hat v}-R_0e^{-v}).$$
 Making some straightforward calculations and using the facts that $\Delta_0$ and $R_0$ are continuous and $v$ is bounded, we deduce that there exists a positive constant $K$ such that for all $\zeta$ close enough to $\hat \zeta$ we have
 $$\frac{\partial\omega}{\partial t}\leq e^{-v}\Delta_0\omega+K\omega+\epsilon.$$
 The maximum principle then implies that $\omega$ is bounded by the solution to the Cauchy problem
 \begin{equation*}
     \left\{\begin{aligned} &\dot\varphi=K\varphi+\epsilon,\\
                            &\varphi(0)=0.
             \end{aligned}
     \right.
 \end{equation*}
 Therefore,  $|v-\hat v|\leq \sqrt{\frac{e^{KT}-1}{K}}\sqrt{\epsilon}$ for all $\zeta$ sufficiently close to $\hat\zeta$, and the lemma is proved.
\end{proof}

\begin{proof}[Proof of Theorem \ref{theorem:continuity}:]
 Since the lamination is hyperbolic we can cover it with covering tubes of the form $U_{\Sigma}\cong\R^2\times\Sigma$. Take an exhaustion of $U_{\Sigma}$
 by open
  sets of the form $\Omega_k\times\Sigma$ with $\Omega_k$ as in Proposition 5.2. Then, for each $\zeta\in\Sigma$ there exists a unique 
  solution $v_k(\cdot,\cdot,\zeta)$ to the problem
\begin{equation*}
 \left\{\begin{aligned}
         &\frac{\partial v}{\partial t}-e^{-v}\tilde{\Delta}_0 v=r-\tilde{R}_0e^{-v}\\
         &v(0,z,\zeta)=0\quad \forall (z,\zeta)\in \Omega_k\times \Sigma\\
         &v(t,z,\zeta)=0\quad \forall (z,\zeta)\in\partial\Omega_k\times\Sigma,
        \end{aligned}\right.
  \end{equation*}
  defined on a positive time interval $[0,T]$ which does not depends on $\zeta$,
  and which belongs to $C^{\infty,0}(\Omega_k\times\Sigma)$ for each $t\in[0,T]$. Applying the maximum principle we have that $v_k$ is uniformly bounded by a 
  constant independent of $k$, in fact we have
  $$\log\left(\frac{R_{max}}{r}\right)\leq v_k(z,\zeta,t)\leq\log\left(\frac{R_{min}}{r}\right),$$
  where $R_{min}$ and $R_{max}$ are as in Section 3.2.
  Moreover, the thesis of Proposition 3.3 holds and the curvature and all its tangential derivatives are uniformly bounded. Therefore, there exists a subsequence
  of $\{v_k\}$ that converges uniformly in any $C^{n,0}-$norm to a function $v\in C^1([0,T],C^{\infty,0}(U_{\Sigma}))$ that satisfies the equation 
  $\frac{\partial v}{\partial t}-e^{-v}\tilde{\Delta}_0 v=r-\tilde{R}_0e^{-v}$ in $U_{\Sigma}$. As was mentioned  before, uniqueness of the Ricci flow equation in each leaf implies that $v$ is the pullback of a function defined on $\Phi(U_{\Sigma})$. Again, by uniqueness, $v$ is necessarily equal to $\Phi^{\ast}(u\big|_{\Phi(U_{\Sigma})})$, then $u\in C^{\infty,0}(M,\Lam)$.
  \end{proof}

\section{Proof of the uniformization theorem.}\label{uniformization}  
  
Now we have all the ingredients to conclude the proof of Theorem 2.1. The argument is the same as
that used by Hamilton in \cite{Ham1,Ham2}.

\begin{proof}[Proof of Theorem \ref{theorem:uniformization}:] Given any conformal class of laminated metrics let $g_0$ be the metric of negative curvature in that class given by Theorem \ref{lemma:etienne}. Let $g(t)=e^{u_t}g_0$ be the solution to the Ricci flow equation 
on the lamination, with initial condition $g_0$, whose existence is assured by Theorem 
\ref{theorem:continuity} as well as its 
$C^{\infty,0}$ regularity. As was mentioned earlier the function $u(x,t)$ is uniformly bounded 
by a constant which is independent of $t$. By virtue of Proposition \ref{proposition:curvature bounds}
the curvature of $g(t)$ converges uniformly to the constant $r$ and all its derivatives converge 
uniformly to zero as $t$ goes to infinity. This, together with the formula
$$u(x,t)=\int_0^t(r-R(x,s))ds,$$
implies that the solution $g(t)$ is defined for all $t\geq0$ and that 
the limit $g_{\infty}=e^{u_{\infty}}g_0=\lim_{t\to\infty}g(t)$ exists and is of class $C^{\infty,0}$. Rescaling the metric we can make the curvature be -1.
\end{proof}

\begin{Rem}
 It would be nice to do without Theorem \ref{lemma:etienne}, and to start the Ricci flow with an arbitrary initial condition having 
 curvature of possibly varying sign. The authors have failed in doing so. The question remains: 
 is it possible for a compact hyperbolic surface lamination to start the flow with an arbitrary initial metric and to prove that the solution converges to the constant curvature metric?
\end{Rem}

\begin{Rem}
While we are not able to prove that there is a trajectory of the Ricci flow from any initial metric $g_0$ to the metric that makes all leaves of constant curvature -1, certainly there is a path which joins these two metrics inside the conformal class of $g_0$. To do that we can first join with a segment the metric $g_0$ to the metric given by Theorem 3.5, and then run the Ricci flow starting at that metric. This proves that the conformal classes are path conected. \end{Rem}
  
\section{Concluding Remarks}\label{Concluding_Remarks}

If all of the leaves of a compact surface lamination $(M,\Lam)$ are of parabolic type, i.e.\ the universal cover of every leaf is conformally equivalent to 
flat $\R^2$ or, equivalently, all the universal covers of the leaves have polynomial growth (and no leaf is diffeomorphic to the 
2-sphere $\mathbb{S}^2$), and we fix the conformal class of a laminated metric $g$ it is not always possible to find a metric in the conformal class of $g$ 
such that every leaf has curvature 0.
\'Etienne Ghys in \cite{Ghys} gave an example of a compact real-analytic 2-dimensional foliation $(M,\Lam)$ with a 
laminated metric $g$ such that
\begin{enumerate}
\item Every leaf is dense and has polynomial growth
\item Every leaf is parabolic
\item It does not exist a $C^{\infty,0}$ function $u:M\to\R$ 
such that with respect to the metric $e^ug$ every leaf is complete and flat.
\end{enumerate}

If all the leaves of a compact lamination are elliptic, then the existence of a uniformizing metric is a direct consequence of Reeb's theorem and Ahlfors-Bers theory (see ~\cite{Ghys}). Since in this case all leaves are diffeomorohic to spheres, Ricci flow equation will also give a proof of the uniformization theorem as a consequence of the validity of the proof for a sphere.

\bibliographystyle{plain}

\end{document}